\documentclass[11pt,psamsfonts]{amsart}
\usepackage{amsmath}
\usepackage{amsthm}
\usepackage{amssymb}
\usepackage{amscd}
\usepackage{amsfonts}
\usepackage{amsbsy}
\usepackage{epsfig,afterpage}
\usepackage[dvips]{psfrag}
\usepackage{color}     % Habilita o uso de cores
\usepackage{fancyhdr}

\newcommand{\R}{\ensuremath{\mathbb{R}}}

\newcommand{\N}{\ensuremath{\mathbb{N}}}

\newcommand{\dis}{\displaystyle}
\newcommand{\vs}{\vspace{0,5cm}}

\newtheorem {theorem} {Theorem} %[section]
\newtheorem {proposition} [theorem] {Proposition}

\newtheorem {lemma} [theorem] {Lemma}
\newtheorem {definition} [theorem] {Definition}
\newtheorem {remark} {Remark}

\begin{document}

\title[Asymptotically stable cusp-fold singularity in 3D] {An asymptotically stable cusp-fold singularity in 3D piecewise smooth vector fields.}

\author[T. de Carvalho, M.A. Teixeira and D.J. Tonon]
{Tiago de Carvalho$^1$, Marco A. Teixeira$^2$\\ and Durval J. Tonon$^3$}

%
%\address{$^1$ IBILCE--UNESP, CEP 15054--000,
%S. J. Rio Preto, S\~ao Paulo, Brazil}

\address{$^1$ FC--UNESP, CEP 17033--360,
Bauru, S\~ao Paulo, Brazil}

\address{$^2$ IMECC--UNICAMP, CEP 13081--970, Campinas,
S\~ao Paulo, Brazil}

\address{$^3$ Universidade Federal de Goi\'{a}s, IME, CEP 74001-970, Caixa
Postal 131, Goi\^{a}nia, Goi\'{a}s, Brazil.}

%\email{buzzi@ibilce.unesp.br}

\email{tcarvalho@fc.unesp.br}

\email{teixeira@ime.unicamp.br}

\email{djtonon@mat.ufg.br}


%\date{}

\dedicatory{} \maketitle

%%%%%%%%%%%%%%%%%%%%%%%%%%%%%%%%%%%%%%%%%%%%%%%%%%%%%%%%%%%%%%%%%%%%%%%%%%%%%%%%%%%%%%%%%%%
%%%%%%%%%%%%%%%%%%%%%%%%%%%%%%%%%%%%%%%%%%%%%%%%%%%%%%%%%%%%%%%%%%%%%%%%%%%%%%%%%%%%%%%%%%%

%\begin{center}{\large Version: \today}\end{center}

\begin{abstract}

This paper is concerned with the analysis of
a typical singularity of piecewise smooth vector fields on $\R^3$ composed by two zones. In our object of study, the cusp-fold
singularity, we consider the simultaneous occurrence of a cusp singularity for one vector field and a fold singularity for the other one. We exhibit a normal form that presents one of the most important property searched for in piecewise smooth vector fields: the asymptotical stability.
\end{abstract}
%%%%%%%%%%%%%%%%%%%%%%%%%%%%%%%%%%%%%%%%%%%%%%%%%%%%%%%%%%%%%%%%%%%%%%%%%%%%%%%%%%%%%%%%%%%
%%%%%%%%%%%%%%%%%%%%%%%%%%%%%%%%%%%%%%%%%%%%%%%%%%%%%%%%%%%%%%%%%%%%%%%%%%%%%%%%%%%%%%%%%%%

\section{Introduction}\label{secao introducao}

In this paper we study piecewise smooth vector fields
(PSVFs for short) $Z$ on $\R^3$. Our goal is to describe the
local dynamics around typical singularities of $Z$
consisting of two smooth vector fields $X, Y$ in $\R^3$ such
that on one side of a smooth surface $\Sigma = \{ z = 0 \}$ we take
$Z = X$ and on the other side $Z = Y$.

PSVFs  are widely used in Electrical and Electronic
Engineering, Physics, Economics, among other areas. In our approach Filippov's convention (see \cite{Fi}) is considered. So, the vector field is discontinuous across the \textbf{switching manifold} $\Sigma$ and it is possible for its trajectories to be
confined onto the switching manifold itself. The occurrence of such
behavior, known as sliding motion, has been reported in a
wide range of applications (see for instance \cite{diBernardo-livro,diBernardo-electrical-systems,Lamb-Makarenkov} and references therein).

The main tool treated here concerns the non transversal contact between a general smooth
vector field and the boundary $\Sigma$ of a manifold. Such points are distinguished
singularities $-$ important objects to be analyzed when one studies Filippov
systems (see \cite{Eu-fold-cusp,Eu-fold-sela,Marcel,Kuznetsov} for a planar analysis on this subject). In the $3$-dimensional case, there are two important distinguished generic singularities: the points where this contact is either quadratic or cubic, which are called \textbf{fold points} and \textbf{cusp points} respectively. As it is fairly known, from a generic cusp point emanate two branches of fold points (see Figure \ref{fig cusp-fold e two-fold inicio}), one of such branches formed by \textbf{visible fold points}, where the trajectories tangent are visible and one of such branches formed by \textbf{invisible fold points}, where the trajectories tangent are not visible. Moreover, it is possible for a point $p \in \Sigma$ be a tangency point for both $X$ and $Y$. When $p$ is a fold point of both $X$ and $Y$ we say
that $p$ is a \textbf{two-fold singularity} and when $p$ is a cusp point for $X$ and a fold point for $Y$ we say that $p$ is a \textbf{cusp-fold singularity} (see Figure \ref{fig cusp-fold e two-fold inicio} below).
\begin{figure}[!h]
\begin{center} \epsfxsize=12cm
\epsfbox{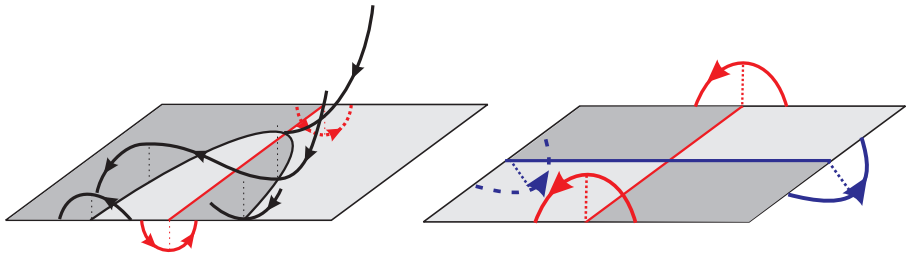}
\caption{\small{On the left it appears a cusp-fold singularity and on the right a two-fold singularity.}} \label{fig cusp-fold e two-fold inicio}
\end{center}
\end{figure}
In \cite{Jeffrey-colombo, Jeffrey-colombo-2011, J-T-T1, J-T-T2} two-fold singularities are studied and their normal forms and
phase portraits are exhibited and in \cite{diBernardo-electrical-systems, Jeffrey-T-sing} are exhibited applications of such theory in electrical and control systems, respectively. This singularity is particularly relevant because in its neighborhood  some of the key
features of a piecewise smooth system are present: orbits that cross
 $\Sigma$, those that slide along it according to
Filippov's convention, among others.

In this paper we analyze the  bifurcation diagram and the asymptotical stability of the following family presenting a fold-cusp singularity:
\begin{equation}\label{eq forma normal cusp-fold inicial com lambda}
 Z_{\lambda}(x,y,z) = \left\{
      \begin{array}{ll}
        X^{\lambda}_{a,b} =  \left(
              \begin{array}{c}
              a \\
              \lambda \\
              b (y+x^2)
\end{array}
      \right)
 & \hbox{if $z \geq 0$,} \\
       Y_{c,d} = \left(
              \begin{array}{c}
               c \\
               d \\
               x
\end{array}
      \right)& \hbox{if $z \leq 0$,}
      \end{array}
    \right.
\end{equation}with $a,b,c,d, \lambda\in \R, b\cdot c\neq 0$ and $\lambda$ is arbitrarily small. Moreover, we observe that the topological dynamic of this still poor studied object is even more
sophisticated than that one exhibited by the two-fold singularity.
In fact, by means of the variation of the parameter $\lambda$ occurs the birth of two-fold
singularities approaching the cusp-fold singularity. %\textcolor{red}{este restante de frase ficou errado diante dos novos resultados}As shown in Section XXX, the \textbf{First Return Map} (see Definition XXX below) associated to \eqref{eq forma normal cusp-fold inicial com lambda} is not hyperbolic at  it singularity situated at the origin when $\lambda=0$ and the origin is a repeller hyperbolic singularity for it if $\lambda < 0$ \marginpar{descrever aqui a dinamica para primeiro retorno}and an attractor hyperbolic singularity for it if $\lambda >0$. This give origin to a "\textit{Like-Hopf bifurcation}" where a closed orbit is observed. However this closed orbit intersects $\Sigma$ at the interior of the sliding region and consequently it is observed only virtually. Another relevant point to stress is that the asymptotical stability of the trajectories of \eqref{eq forma normal cusp-fold inicial com lambda} is established only when $\lambda \geq 0$, being asymptotically unstable when $\lambda<0$.

The paper is organized as follows: In Section \ref{secao teoria basica} we formalize some basic concepts on PSVFs and the concept of first return map in this scenario is formalized. In Section \ref{secao resultados principais} the problem is described, the main results are stated and we pave the way in order to prove the main results in Section \ref{main results}.

%Besides, we wrote a section called Appendix were we deal with equivalence classes on PSVFs.

%%%%%%%%%%%%%%%%%%%%%%%%%%%%%%%%%%%%%%%%%%%%%%%%%%%%%%%%%%%%%%%%%%%%%%%%%%%%%%%%%%%%%%%%%%%%%%%%%%%%%%%
%%%%%%%%%%%%%%%%%%%%%%%%%%%%%%%%%%%%%%%%%%%%%%%%%%%%%%%%%%%%%%%%%%%%%%%%%%%%%%%%%%%%%%%%%%%%%%%%%%%%%%%

\section{Basic Theory about PSVF's}\label{secao teoria basica}

Let $K = \{ (x,y,z) \in \R^3 \, | \, x^2 + y^2 + z^2 < \delta  \}$,
where $\delta>0$ is arbitrarily small. Consider $\Sigma = \{ (x,y,z)
\in K \, | \, z=0 \}$. Clearly the switching manifold
$\Sigma$ is the separating boundary of the regions
$\Sigma_+=\{(x,y,z) \in K \, | \, z \geq 0\}$ and $\Sigma_-=\{(x,y,z) \in
K \, | \, z \leq 0\}$.

Designate by $\chi^r$ the space of $C^r$-vector fields on $K$
endowed with the $C^r$-topology with $r=\infty$ or $r\geq 1$ large
enough for our purposes. Call \textbf{$\Omega^r$} the
space of vector fields $Z: K \rightarrow \R ^{3}$ such that
\begin{equation}\label{eq Z}
 Z(x,y,z)=\left\{\begin{array}{l} X(x,y,z),\quad $for$ \quad (x,y,z) \in
\Sigma_+,\\ Y(x,y,z),\quad $for$ \quad (x,y,z) \in \Sigma_-,
\end{array}\right.
\end{equation}
where $X=(X_1,X_2,X_3)$ and $Y = (Y_1,Y_2,Y_3)$ are in $\chi^r.$ We
may consider $\Omega^r = \chi^r \times \chi^r$ endowed with the
product topology  and denote any element in $\Omega^r$ by $Z=(X,Y),$
which we will accept to be multivalued in points of $\Sigma$. The basic results of
differential equations, in this context, were stated by Filippov in
\cite{Fi}. Related theories can be found in \cite{diBernardo-livro, Marco-enciclopedia} and references therein. On $\Sigma$ we generically distinguish the following regions:
\begin{itemize}
\item \textbf{Crossing Region:} $\Sigma^c=\{ p \in \Sigma \, | \, X_3(p).Y_3(p)> 0 \}$.
Moreover, we denote $\Sigma^{c+}= \{ p \in \Sigma \, | \,
X_3(p)>0,Y_3(p)>0 \}$ and $\Sigma^{c-} = \{ p \in \Sigma \, | \,
X_3(p)<0,Y_3(p)<0 \}$.

\item  \textbf{Sliding Region:} $\Sigma^{s}= \{ p \in \Sigma \, | \, X_3(p)<0,
Y_3(p)>0 \}$.

\item \textbf{Escaping
Region:} $\Sigma^{e}= \{ p \in \Sigma \, | \, X_3(p)>0
,Y_3(p)<0\}$.
\end{itemize}

When $q \in \Sigma^s$, following the Filippov's convention, the \textbf{sliding vector field}
associated to $Z\in \Omega^r$ is the vector field  $\widehat{Z}^s$ tangent to
$\Sigma^s$ expressed in coordinates as
\begin{equation}\label{eq campo filippov}\widehat{Z}^s(q)= \frac{1}{(Y_3 - X_3)(q)} ((X_1 Y_3 - Y_1 X_3)(q),(X_2 Y_3 - Y_2 X_3)(q),0).\end{equation}
Associated to \eqref{eq campo filippov} there exists the planar \textbf{normalized sliding vector field}
\begin{equation}\label{equacao campo normalizado}Z^{s}(q)=((X_1 Y_3 - Y_1 X_3)(q),(X_2 Y_3 - Y_2 X_3)(q)).
\end{equation}Observe that $\widehat{Z}^{s}$ and $Z^{s}$ are topologically
equivalent in $\Sigma^s$ and $Z^s$ can be C$^r$-extended to the closure $\overline{\Sigma^s}$ of $\Sigma^s$.\\

\begin{lemma}\label{lema equilibrio sliding}
Given $Z=(X,Y) \in \Omega^r$ if $q \in \Sigma$ is a two-fold or a cusp-fold singularity of $Z$, then $q$ is an equilibrium point of the normalized sliding vector field given in \eqref{equacao campo normalizado}.
\end{lemma}
\begin{proof}
It is straightforward since both $X$ and $Y$ are tangent to $\Sigma$ at $q$. So, $X_3(q)=Y_3(q)=0$ and $Z^{s}(q)=(0,0)$.
\end{proof}

In fact, the previous lemma remains true when the trajectories of both $X$ and $Y$ have a non transversal contact point at $q$ regardless the order of such contact.

The points $q \in \Sigma$ such that $Z^s(q)=0$ are called \textbf{pseudo equilibria of $\mathbf{Z}$} and the points $p \in \Sigma$ such that $X_3(p).Y_3(p) =0$ are called \textbf{tangential singularities of $\mathbf{Z}$}  (i.e., the trajectory through $p$ is tangent to $\Sigma$). Note that two-fold and cusp-fold singularities are both pseudo equilibria and tangential singularities.

\vspace{.5cm}
\noindent \textbf{Notations:}
\begin{itemize} \item We denote the flow of a vector field $W \in \chi^r$ by $\phi_{W}(t,p)$ where $t \in I$
with $I=I(p,W)\subset \R$ being an interval depending on $p \in K$
and $W$.

\item Given a vector field $W$ defined in $A \subset K$, we denote the \textbf{backward trajectory} $\phi^{-}_{W}(A)$ (respectively, \textbf{forward trajectory} $\phi^{+}_{W}(A)$) the set of all negative (respectively,
positive)  orbits of $W$ through points of $A$.

\item We denote the boundary of an arbitrary set $A \subset K$ by $\partial A$.
    \end{itemize}

Following \cite{Marcel}, page 1971, we consider the definition:

\begin{definition}\label{definicao trajetorias}
The forward local trajectory $\phi^{+}_{Z}(t,p)$ of a PSVF given by \eqref{eq
Z} through $p \in \Sigma$ is defined as follows:
\begin{itemize}
\item[(i)] $\phi^{+}_{Z}(t,p)=\phi_{X}(t,p)$ (respectively, $\phi_{Z}(t,p)=\phi_{Y}(t,p)$) provided that $p \in \Sigma^{c+}$ (respectively, $p \in \Sigma^{c-}$).

\item[(ii)] $\phi^{+}_{Z}(t,p)=\phi_{Z^{s}}(t,p)$ provided that $p \in
\Sigma^s$.

\item[(iii)] splits in two orbits  $\phi^{+}_{Z}(t,p)=\phi_{X}(t,p)$
and $\phi^{+}_{Z}(t,p)=\phi_{Y}(t,p)$ provided that $p \in \Sigma^{e}$.

\item[(iv)] For $p \in \partial \Sigma^s \cup \partial \Sigma^e \cup
\partial \Sigma^c$ such that the definitions of forward trajectories  for points
in a full neighborhood of $p$ in $\Sigma$ can be extended to $p$
and coincide, the trajectory through $p$ is this trajectory.

\item[(v)] For any other point $\phi^{+}_{Z}(t,p)=p$ for all $t \in \R$.
\end{itemize}\end{definition}

\subsection{The fist return map}

The following construction is presented  in
\cite{Marco-enciclopedia}. Let $Z=(X,Y) \in \Omega^r$ such that $q$
is an invisible fold point of $X$. From Implicit Function
Theorem, for each $p \in \Sigma$ in a neighborhood $\mathcal{V}_{q}$
of $q$ there exists a unique $t(p) \in
(-\delta,\delta)$, a small interval,  such that $\phi_X(t,p)$  meets $\Sigma$ in
$\widetilde{p}=\phi_X(t(p),p)$. Define the
 map $\gamma_X:\mathcal{V}_{q} \cap \Sigma \rightarrow
\mathcal{V}_{q} \cap \Sigma$ by $\gamma_X(p)=\widetilde{p}$. This
map is a $C^{r}$-diffeomorphism and satisfies: $\gamma_{X}^{2}=Id$.
Analogously, when $\widetilde{q}$ is an invisible fold
point of $Y$  we define the  map
$\gamma_Y:\mathcal{V}_{\widetilde{q}} \cap \Sigma \rightarrow
\mathcal{V}_{\widetilde{q}} \cap \Sigma$ associated to $Y$ which
satisfies $\gamma_{Y}^{2}=Id$. Now we can give the following definition:

\begin{definition}\label{definicao aplicacao primeiro retorno}
Let $\mathcal{T} \subset \Sigma^{c}$ be an open region of $\Sigma$. The \textbf{First Return Map} $\varphi_Z:
\mathcal{T} \rightarrow \mathcal{T}$ is defined by the composition
$\varphi_Z=\gamma_Y \circ \gamma_X$ when $\gamma_X (\mathcal{T})
\subset \Sigma^{c}$ and $\gamma_X,\gamma_Y$
are well defined in $\mathcal{T}, \gamma_X (\mathcal{T})$ respectively.\end{definition}

Considering the PSVF given in \eqref{eq forma normal cusp-fold inicial com lambda}, we get the expression of the first return map
\begin{equation}
\varphi_{Z_{\lambda}}(x,y)=\left(\frac{  2 a x + \Delta_1}{4a}, y + \frac{d(2 a x + \Delta_1)}{2ac} + \frac{\lambda( - 6 a x - \Delta_1)}{4a^2}\right)
\label{aplicacao-primeiro-retorno}\end{equation}where $\Delta_1=3 \lambda - \sqrt{9 \lambda^2 + 36 a \lambda x - 12 a^2 (x^2 + 4 y)}$.

Note that we can extend $\varphi_{Z_{\lambda}}$ to the boundary of $SwR$. In this way, the unique fixed point of $\varphi_{Z_{\lambda}}$, in a neighborhood of origin, is the origin. Let
\[
\Delta_2=(ad)^2-adc\lambda.
\] When $\lambda\neq 0$, the eigenvalues of $D\varphi_{Z_{\lambda}}$ at origin are
\begin{equation}
\xi_{\pm}^{\lambda}=\frac{2ad-c\lambda\pm 2\sqrt{\Delta_2}}{c\lambda},
\label{expressao-autovalor-primeiro-retorno}\end{equation}the eigenvectors associated to $\xi_+^{\lambda}$ and $\xi_-^{\lambda}$ respectively, are
\[
v_{\pm}^{\lambda}=(\omega_{\pm}^{\lambda},1),
\]where $\omega_{\pm}^{\lambda}=\frac{ac}{ad\pm \sqrt{\Delta_2}}$ and
the eigenspaces associated to $\xi_{\pm}^{\lambda}$, respectively, are tangent to the straight lines
\begin{equation}
S_{\pm}^{\lambda}=\left\{(x,y,0)\in \Sigma|x=\dis\frac{ac}{ad\pm \sqrt{\Delta_2}} y\right\}.
\label{expressao-autoespaco-primeiro-retorno}\end{equation}

%\textcolor{red}{Eh essencial para a linearidade do texto repetir a analise anterior para lambda=0 aqui. Fazer figuras dos retratos de fase das aplicacoes de primeiro retorno aqui.}

\section{Main Results}\label{secao resultados principais}

The main results of the paper are now stated. \\

\noindent\textbf{Theorem A.}
Let $Z_{\lambda}$ given by \eqref{eq forma normal cusp-fold inicial com lambda} presenting a fold-cusp singularity. If $a<0$, $b<0$, $c>0$, $d>0$ and $a+ bd>0$ then:
\begin{itemize}
\item $Z_{\lambda}$ is asymptotically stable when $\lambda \geq 0$ and

\item $Z_{\lambda}$ is not asymptotically stable when $\lambda < 0$.
\end{itemize}

%\vspace{.4cm}
%
%\noindent\textbf{Theorem B.} Algum outro resultado ? codimensao do primeiro retorno e consequentemente, codimensao da fold-cusp ? (explicaria pq nao é do tipo Q3-sing de cod 1)

%%%%%%%%%%%%%%%%%%%%%%%%%%%%%%%%%%%%%%%%%%%%%%%%%%%%%%%%%%%%%%%%%%%%%%%%%%%%%%%%%%%%%%%%%%%%%%%%%%%%%%%
%%%%%%%%%%%%%%%%%%%%%%%%%%%%%%%%%%%%%%%%%%%%%%%%%%%%%%%%%%%%%%%%%%%%%%%%%%%%%%%%%%%%%%%%%%%%%%%%%%%%%%%

\subsection{Proof for the case $\lambda=0$}\label{subsecao-lambda-zero} Let \eqref{eq forma normal cusp-fold inicial com lambda} with $\lambda=0$, i.e., the following normal form presenting a cusp-fold singularity at the origin:

\begin{equation}\label{eq forma normal inicial lambda zero}
 Z_{0} (x,y,z) = \left\{
      \begin{array}{ll}
        X^{0}_{a,b} =  \left(
              \begin{array}{c}
              a \\
              0 \\
              b (y+x^2)
\end{array}
      \right)
 & \hbox{if $z \geq 0$,} \\
       Y_{c,d} = \left(
              \begin{array}{c}
               c \\
               d \\
               x
\end{array}
      \right)& \hbox{if $z \leq 0$.}
      \end{array}
    \right.
\end{equation}

Note that $S_{X} = \{ (x,y,z) \in \Sigma \, | \, y=-x^2 \}$ and  $S_{Y} = \{ (x,y,z) \in \Sigma \, |
\, x=0
\}$ are the sets of tangential singularities of $X$ and $Y$ respectively.

\subsubsection{Local dynamics of the normalized sliding vector fields}

Using \eqref{equacao campo normalizado}, the normalized sliding vector field is given by
$$
Z_{0}^{\Sigma} = (a x - b c (y+x^2), - d b (y+x^2)).
$$
So, the eigenvalues of $Z_{0}^{\Sigma}$ are
$$
\lambda^{0}_1 = a \mbox{ and } \lambda^{0}_2 = -db,
$$
the eigenvectors associated to $\lambda^{0}_1$ and $\lambda^{0}_2$ respectively, are
$$
v^{0}_1 = \left(
        \begin{array}{c}
          1 \\
          0 \\
        \end{array}
      \right) \mbox{ and }
      v^{0}_2 = \left(
        \begin{array}{c}
          \frac{bc}{a+bd} \\
          1 \\
        \end{array}
      \right)
$$
and the eigenspaces associated to $\lambda_1$ and $\lambda_2$ respectively, are
\begin{equation}\label{autoespacos}
E^{0}_1 = \{  (x,y,0) \in \Sigma \, | \, y=0  \}
\mbox{ and }
E^{0}_2 =  \Big\{  (x,y,0) \in \Sigma \, | \, y= \frac{(a + bd) x}{bc}  \Big\}.
\end{equation}

In order to obtain a cusp-fold singularity asymptotically stable some hypotheses must be imposed in the parameters.\\

\noindent{\textbf{Hypothesis 1 ($H_1$):}} The fold point generated by the vector field $Y$ must be invisible. So, $c >0$.

%
%
%\begin{figure}[ht]
%\epsfysize=3.5cm \psfrag{a}{$\Sigma^{c-}$}\psfrag{b}{$\Sigma^{c+}$}\psfrag{c}{$\Sigma^{e}$}\psfrag{d}{$\Sigma^{s}$}\psfrag{S}{$\Sigma$}
%\centerline{\epsfbox{dobra-cuspide-Lzero.eps}}\caption{The local dynamics of $Z_0$ with hypothesis $H_1$ and $H_3$.}\label{dobra-cuspide-Lzero}
%\end{figure}

\noindent{\textbf{Hypothesis 2 ($H_2$):}} The origin must be asymptotically stable for $Z^{\Sigma}$. So $$\lambda^{0}_1 = a <0 \mbox{ and } - \lambda^{0}_2 = bd >0.$$

Note that, since $a<0$, the vector field $X^{0}$ goes from the right to the left. This implies that the vector field $Y$ goes from the left to the right in order to permit recurrences. In fact this happens because $c>0$, according to $H_1$.

Following $H_1$ and $H_2$, the phase portrait of $Z$, in $\Sigma^s$, is given by one of the following illustrations, in Figure \ref{fig-duas-cusp-fold}:

\begin{figure}[ht]
\epsfysize=3.5cm \psfrag{A}{$\Sigma^{c-}$}\psfrag{B}{$\Sigma^{c+}$}\psfrag{C}{$\Sigma^{e}$}\psfrag{D}{$\Sigma^{s}$}\psfrag{S}{$\Sigma$}
\centerline{\epsfbox{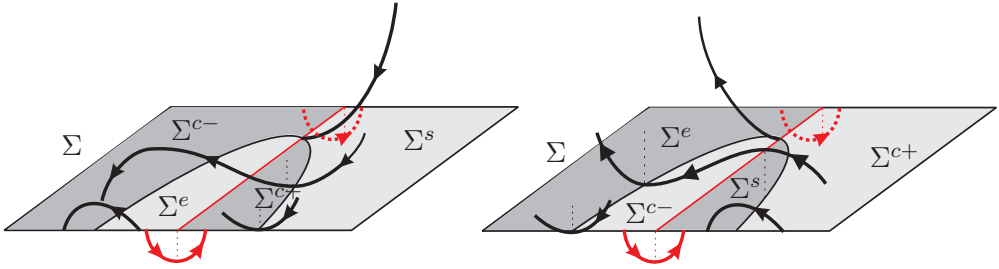}}\caption{The two possible local dynamics of $Z_0$ with hypothesis $H_1$ and $H_2$.}\label{fig-duas-cusp-fold}
\end{figure}

However, as can be easily checked, just at the case $(a)$ of Figure \ref{fig-duas-cusp-fold} we hope some asymptotical stability. So we consider the next hypothesis:\\

\noindent{\textbf{Hypothesis 3 ($H_3$):}} The cusp singularity generated by the vector field $X$ must be of the topological type described in Figure \ref{fig-duas-cusp-fold}. So, $$b <0.$$

By consequence of $H_2$ and $H_3$ we conclude that $d<0$.  Moreover, in $\Sigma^s$ we get $x>0$ and $y+x^2>0$, so $Y_3 - X_3 = x -b (y + x^2)>0$ and the sliding vector field in \eqref{equacao campo normalizado} has the same orientation of \eqref{eq campo filippov}.

\begin{lemma}\label{lema tangencia E1}
The eigenspace $E_2^0$ associated to $\lambda_2$ is tangent to the curve $S_X$, in $\Sigma$.
\end{lemma}
\begin{proof}
Straightforward according to \eqref{autoespacos}.
\end{proof}

Faced to $H_2$, in order to obtain that the sliding region $\Sigma^s$ is invariant for $Z_0^{\Sigma}$ and the origin is asymptotically stable, we impose the following hypothesis:

\noindent{\textbf{Hypothesis 4 ($H_4$):}} $E_2^0$ is stronger than $E_1^0$, i.e., $| \lambda_1| < |\lambda_2|$. So,
$$
-a < bd \Rightarrow 0 < a + bd.
$$

As an immediate consequence of $H_4$, we get $(bc)/(a+ bd) <0$ and $E_2^0 \cap \Sigma^s = \emptyset$ (so, in fact, $\Sigma^s$ is an invariant for $Z_0^{\Sigma}$). See Figure \ref{fig deslize H4}.\\
\begin{figure}[ht]
\epsfysize=2cm \psfrag{A}{$\Sigma^{c-}$}\psfrag{B}{$\Sigma^{c+}$}\psfrag{C}{$\Sigma^{e}$}\psfrag{D}{$\Sigma^{s}$}\psfrag{S}{$\Sigma$} \psfrag{1}{Case $(a)$}\psfrag{2}{Case $(b)$}\psfrag{E1}{$E^{0}_1$}\psfrag{E2}{$E^{0}_2$}
\centerline{\epsfbox{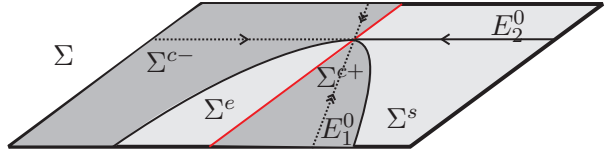}}\caption{Local dynamics of $Z^{\Sigma}$.}\label{fig deslize H4}
\end{figure}
\subsubsection{Local dynamics of the first return map}
Now, in order to determine the dynamics of the positive trajectories of $Z_0$ we consider the First Return Map, given in \eqref{aplicacao-primeiro-retorno}, with $\lambda=0$. We get
$$
\varphi_{Z_0}(x,y) = \left( \frac{a x - \sqrt{-3 a^2 (x^2 + 4 y)}}{2 a} , y + \frac{d (a x - \sqrt{-3 a^2 (x^2 + 4 y)} }{ a c}  \right).
$$

Given a point $p \in \R^3$, it is easy to see that the positive trajectory $\phi^{+}_{Z_0}(p)$ of $Z$ passing through $p$ intersects $\overline{\Sigma^s} \cup \overline{\Sigma^{c+}}$. In what follows we prove that $\phi^{+}_{Z_0}(p) \cap \overline{\Sigma^s} \neq \emptyset$ and, after an appropriated choice on the parameters $a,b,d$, we obtain that  $\phi^{+}_{Z_0}(p)$ converges to the origin when $t\rightarrow +\infty$.

\begin{lemma}\label{lema imagem da parabola}
The image of the curve $y=-x^2$, with $x>0$, by the First Return Map $\varphi_{Z_0}$ is the curve $y= - \frac{x^2}{4} + 2 \frac{d}{c} x$ with $x>0$, i.e.,
$$
\varphi_{Z_0}(\{ y=-x^2 \mbox{ with } x>0 \}) = \Big\{ y= - \frac{x^2}{4} + 2 \frac{d}{c} x \mbox{ with } x>0 \Big\}.
$$
\end{lemma}
\begin{proof} Consider the point $P_0=(x_0, -x_{0}^{2},0)$, with $x_0 >0$. The trajectory of $X_0$ by $P_0$ intersects $\Sigma$ at $P_1=(- 2 x_0, -x_{0}^{2},0)$ after a time $t_1= - 3 x_0/a$. The trajectory of $Y_0$ by $P_1$ intersects $\Sigma$ at $P_2=( 2 x_0, 4 d x_0/c - x_{0}^{2},0)$ after a time $t_2= 4 x_0/c$. Considering the chance of variables $x=2 x_0$, after a time $\overline{t} = t_1 + t_2 = \frac{(4a - 3 c) x_0}{ac}$, the curve $y=-x^2$ return to $\Sigma$ at the curve $y= - \frac{x^2}{4} + 2 \frac{d}{c} x$.
\end{proof}

\begin{lemma}\label{lema imagem do eixo y negativo}
The image of the curve $x=0$, with $y<0$, by the First Return Map $\varphi_{Z_0}$ is the curve $y = - \frac{x^2}{3} + 2 \frac{d}{c} x$ with $x>0$, i.e.,
$$
\varphi_{Z_0}(\{ x=0 \mbox{ with } y<0 \}) = \Big\{ y= - \frac{x^2}{3} + 2 \frac{d}{c} x \mbox{ with } x>0 \Big\}.
$$
\end{lemma}
\begin{proof}Consider the point $P_0=(0, y_{0}^{2},0)$, with $y_0 <0$. The trajectory of $X_0$ by $P_0$ intersects $\Sigma$ at $P_1=(- \sqrt{ - 3 y_0}, -y_0,0)$ after a time $t_1=  - \frac{\sqrt{ - 3 y_0}}{a}$. The trajectory of $Y_0$ by $P_1$ intersects $\Sigma$ at $P_2=( \sqrt{ - 3 y_0}, \frac{ 2 \sqrt{ - 3 y_0}}{c},0)$ after a time $t_2= \frac{2 \sqrt{ - 3 y_0}}{c}$. Considering the change of variables $x= \sqrt{ - 3 y_0}$, after a time $\overline{t} = t_1 + t_2 = \frac{(2a - c) \sqrt{ - 3 y_0}}{ac}$, the curve $x=0$ return to $\Sigma$ at the curve $y= y= - \frac{x^2}{3} + 2 \frac{d}{c} x$.
\end{proof}

\begin{lemma}\label{lema os pontos da costura com x positivo caem na regiao entre as curvas} The image of the set $\Sigma^{c+}$, by the First Return Map $\varphi_{Z_0}$ remains between the curves $y = - \frac{x^2}{3} + 2 \frac{d}{c} x$ and $y= - \frac{x^2}{4} + 2 \frac{d}{c} x$ with $x>0$, i.e.,
$$
\varphi_{Z_0}(\Sigma^{c+}) \subset \Big\{ (x,y,0) \in \Sigma \, | \, - \frac{x^2}{3} + 2 \frac{d}{c} x < y < - \frac{x^2}{4} + 2 \frac{d}{c} x \mbox{ with } x>0 \Big\}.
$$
\end{lemma}
\begin{proof} Given a point $P_0=(x_0,y_0,0) \in \Sigma^{c+}$ (where $x_0>0$ and $y_0<0$), it is easy to see that the trajectory of $X_0$ by $P_0$ intersects $\Sigma$ at $P_1 \in \Sigma^{c-}$ and the trajectory of $Y_0$ by $P_1$ intersects $\Sigma$ at $P_2$, where $P_2$ is situated between the curves $y = - \frac{x^2}{3} + 2 \frac{d}{c} x$ and $y= - \frac{x^2}{4} + 2 \frac{d}{c} x$ which correspond to the images of the curves  $x=0$, with $y<0$ and $y=-x^2$, with $x>0$, respectively.
\end{proof}

\begin{lemma}\label{lema coordenada x do retorno dos pontos da costura crescem}
Given $p_0=(x_0,y_0,0) \in \overline{\Sigma^{c+}}$, call $p_1=(x_1,y_1,0)=\varphi_{Z_0}(p_0)$ and $p_n=(x_n,y_n,0)=\varphi^{n}_{Z_0}(p_0)$, when it is well defined. Then $x_1>x_0$ and $x_n \rightarrow \infty$ when $n\rightarrow \infty$.
\end{lemma}
\begin{proof} Given $p_0=(x_0,y_0,0) \in \overline{\Sigma^{c+}}$, a straightforward calculus show that $x_1= \frac{x_0}{2} + \frac{\sqrt{3}\sqrt{-(x_{0}^{2} + 4 y_0)}}{2}$ where $p_1=(x_1,y_1,0)=\varphi_{Z_0}(p_0)$. Since $p_0 \in \overline{\Sigma^{c+}}$ we conclude that $y_0 \leq - x_{0}^{2} < - x_{0}^{2}/3$. So,
$$
y_0  < - x_{0}^{2}/3 \Rightarrow - 4 x_{0}^{2}- 12 y_0>0 \Rightarrow 3 (-( x_{0}^{2} + 4 y_0)) > x^2_0 \Rightarrow $$$$\frac{\sqrt{3} \sqrt{-( x_{0}^{2} + 4 y_0)}}{2} > \frac{x_0}{2} \Rightarrow x_1 > x_0.
$$
If $p_1 \in \overline{\Sigma^{s}}$ then a First Return Map is not defined. Otherwise we repeat the previous argument. A recursive analysis prove that $x_{n+1} > x_{n}$. In fact, repeating the previous argument
$$
x_{n+1} = \frac{x_n + \sqrt{3} \sqrt{-( x_{n}^{2} + 4 y_n)} }{2} > 2 x_n \Rightarrow   \frac{x_{n+1}}{x_n} > 2.
$$
Moreover, $\frac{x_{n+1}}{x_{n}}>1$ implies, by a test of convergence of sequences, that $x_n \rightarrow \infty$.
\end{proof}

\begin{proposition}\label{proposicao a trajetoria cai no sliding}
For all $p \in \R^3$ it happens $\phi^{+}_{Z_0}(p) \cap \overline{\Sigma^s} \neq \emptyset$.
\end{proposition}
\begin{proof}
As we observed above, given a point $p \in \R^3$, it is easy to see that  $\phi^{+}_{Z_0}(p) \cap [\overline{\Sigma^s} \cup \overline{\Sigma^{c+}}] \neq \emptyset$. So, it is enough to prove that $\varphi^{n_0}_{Z_0}(\overline{\Sigma^{c+}}) \subset \overline{\Sigma^s}$ for some $n_0>0$.

By Lemmas \ref{lema imagem da parabola}, \ref{lema imagem do eixo y negativo} and \ref{lema os pontos da costura com x positivo caem na regiao entre as curvas} we obtain that
$$\varphi_{Z_0}(\overline{\Sigma^{c+}}) \subset \Big\{ (x,y,0) \in \Sigma \, | \, \frac{x^2}{3} + 2 \frac{d}{c} x \leq y \leq - \frac{x^2}{4} + 2 \frac{d}{c} x \mbox{ with } x>0 \Big\}.$$

By Lemma \ref{lema coordenada x do retorno dos pontos da costura crescem}, there exists $n_0>0$ such that $p_{n_0}=(x_{n_{0}},y_{n_{0}},0)=\varphi^{n_{0}}_{Z_0}(p)$ satisfies $y_{n_{0}} + x_{n_{0}}^2 \geq 0$. Therefore $p_{n_{0}} \in \overline{\Sigma^s}$.
\end{proof}

\subsection{Proof of the case $\lambda \neq0$}\label{secao prova caso lambda diferente de 0}

When $\lambda\neq 0$, we consider the normal form \eqref{eq forma normal cusp-fold inicial com lambda}, presenting a fold-fold singularity at the origin, since $b\neq0$. The local dynamics for $Z_{\lambda}$ is given in Figure \ref{dobra-cuspide-L-pos-neg}. The tangential sets $S_{X}$ and $S_{Y}$ remains the same as the ones established in Subsection \ref{subsecao-lambda-zero}.

\begin{figure}[ht]
\epsfysize=3.7cm \psfrag{a}{$\Sigma^{c-}$}\psfrag{b}{$\Sigma^{c+}$}\psfrag{c}{$\Sigma^{e}$}\psfrag{d}{$\Sigma^{s}$}\psfrag{S}{$\Sigma$} \psfrag{1}{Case $\lambda<0$}\psfrag{2}{Case $\lambda>0$}\psfrag{E1}{$E^{\lambda}_1$}\psfrag{E2}{$E^{\lambda}_2$}
\centerline{\epsfbox{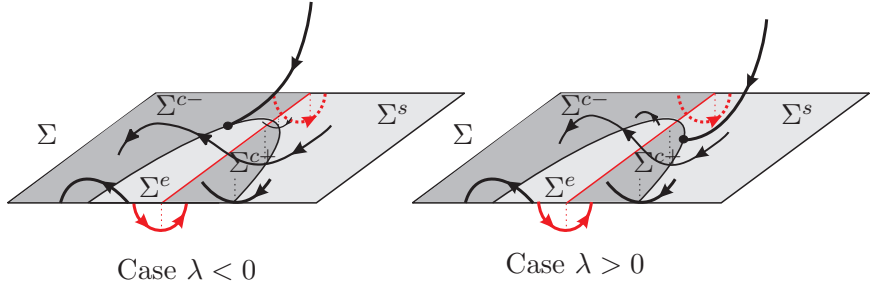}}\caption{The local dynamics of $Z_{\lambda}$ with hypothesis $H_1$ and $H_3$.}\label{dobra-cuspide-L-pos-neg}
\end{figure}

\subsection{Local dynamics of the normalized sliding vector fields}

According to \eqref{equacao campo normalizado}, the normalized sliding vector field is given by
\begin{equation}\label{equacao campo normalizado com lambda}
Z_{\lambda}^{\Sigma} = (a x - b c (y+x^2), \lambda x - d b (y+x^2)).
\end{equation}
Let
\[
\Delta_3=(a + bd)^2 - 4 bc \lambda.
\] So, the eigenvalues of $DZ_{\lambda}^{\Sigma}(0,0)$ are
$$
\lambda^{\lambda}_1 = \frac{a - bd - \sqrt{\Delta_3} }{2} \mbox{ and } \lambda^{\lambda}_2 = \frac{a - bd + \sqrt{\Delta_3}}{2},
$$
the eigenvectors associated to $\lambda_1^{\lambda}$ and $\lambda_2^{\lambda}$ respectively, are
$$
v^{\lambda}_1 = \left(
        \begin{array}{c}
          \frac{a + bd - \sqrt{\Delta_3}}{2 \lambda} \\
          1 \\
        \end{array}
      \right) \mbox{ and }
      v^{\lambda}_2 = \left(
        \begin{array}{c}
          \frac{a + bd + \sqrt{\Delta_3} }{2 \lambda} \\
          1 \\
        \end{array}
      \right)
$$
and the eigenspaces associated to $\lambda_1$ and $\lambda_2$ respectively, are

\begin{equation}\label{autoespacos com lambda}
\begin{array}{ll}  E^{\lambda}_1    &= \left\{  (x,y,0) \in \Sigma \, | \, y= \frac{2 \lambda}{a + bd - \sqrt{\Delta_3}} x  \right\}\\\\
                   E^{\lambda}_2    &= \left\{  (x,y,0) \in \Sigma \, | \, y= \frac{2 \lambda}{a + bd + \sqrt{\Delta_3}} x  \right\}.
\end{array}\end{equation}
Under the hypotheses $H_1,\dots, H_4$, we get the following results:

\begin{lemma}\label{lema tangencia E2}
The eigenspace $E^{\lambda}_1\subset \Sigma^c$ and
\begin{itemize}
\item[(a)] $E^{\lambda}_2 \subset [\Sigma^s\cup \Sigma^e]$ since $\lambda > 0$ and

\item[(b)] $E^{\lambda}_2 \subset \Sigma^c$ since $\lambda < 0$, see Figure \ref{dobra-cuspide-deslizante-Lpos-neg}.
\end{itemize}\end{lemma}
\begin{proof}
Straightforward according to \eqref{autoespacos com lambda}.
\end{proof}

\begin{figure}[ht]
\epsfysize=2.5cm \psfrag{a}{$\Sigma^{c-}$}\psfrag{b}{$\Sigma^{c+}$}\psfrag{c}{$\Sigma^{e}$}\psfrag{d}{$\Sigma^{s}$}\psfrag{S}{$\Sigma$} \psfrag{1}{Case $\lambda<0$}\psfrag{2}{Case $\lambda>0$}\psfrag{E1}{$E^{\lambda}_1$}\psfrag{E2}{$E^{\lambda}_2$}
\centerline{\epsfbox{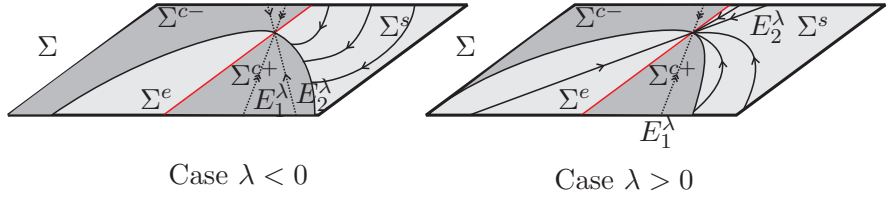}}\caption{The local dynamics of $Z_{\lambda}^{\Sigma}$ with hypothesis $H_1,\dots, H_4$.}\label{dobra-cuspide-deslizante-Lpos-neg}
\end{figure}

Note that in case $\lambda<0$, the sliding vector fields has a transient behavior in $\Sigma^s$, i.e., all the obits in $\Sigma^s$ will be iterated by the first return map, whereas in case $\lambda>0$, $Z_{\lambda}^{\Sigma}$ is asymptotic stable at origin.

\subsubsection{Local dynamics of the first return map}

Now, in order to determine the dynamics of the positive trajectories of $Z_{\lambda}$ we consider the First Return Map $\varphi_{Z_{\lambda}}$ of $Z_{\lambda}$ whose expression is given in \eqref{aplicacao-primeiro-retorno}. %We divide the study in two cases: $\lambda>0$ and $\lambda<0$.

%\subsubsection{Case $\lambda > 0$.}

\begin{lemma}\label{lema dinamica-primeiro-retorno-Ldifzero}
Under the hypothesis $H_1, \dots, H_4$ the origin is a hyperbolic saddle fixed point for $\varphi_{Z_{\lambda}}$ and
\begin{itemize}
\item [$(a)$] $S_{\pm}^{\lambda}\subset \Sigma^c$ since $\lambda>0$
\item [$(b)$] $S_+^{\lambda}\subset \Sigma^c, S_-^{\lambda}\subset [\Sigma^e\cup \Sigma^s]$ since $\lambda<0$.
\end{itemize}
Besides, $S_+^{\lambda}, S_-^{\lambda}$ is a expansive, contraction direction, respectively.
\end{lemma}

\begin{proof} Follows by the expressions \eqref{expressao-autovalor-primeiro-retorno} and \eqref{expressao-autoespaco-primeiro-retorno}, of the eigenvalues and the eigenspaces of $D\varphi_{Z_{\lambda}}(0)$, respectively.
\end{proof}

By Lemma \ref{lema dinamica-primeiro-retorno-Ldifzero}, in case $\lambda>0$, we get that given $p\in \Sigma^{c+}$ there exists $n_0\in \N$ such that $\varphi_{Z_{\lambda}}^{n_0}(p)\in \Sigma^{s}$. And the Lemma \ref{lema tangencia E2}, under the hypothesis $H_1, \dots, H_4$, provides that $Z_{\lambda}^{\Sigma}$ is asymptotic stable at origin. See Figure \ref{dobra-cuspide-primeiro-deslizante-Lpos}, when the orbits in red represent the iterated of $\varphi_{Z_{\lambda}}$ and the orbits in blue the dynamic of $Z_{\lambda}^{\Sigma}$.

\begin{figure}[ht]
\epsfysize=2cm \psfrag{a}{$\Sigma^{c-}$}\psfrag{b}{$\Sigma^{c+}$}\psfrag{c}{$\Sigma^{e}$}\psfrag{d}{$\Sigma^{s}$}\psfrag{S}{$\Sigma$} \psfrag{1}{Case $\lambda<0$}\psfrag{2}{Case $\lambda>0$}\psfrag{E1}{$E^{\lambda}_1$}\psfrag{E2}{$E^{\lambda}_2$}
\centerline{\epsfbox{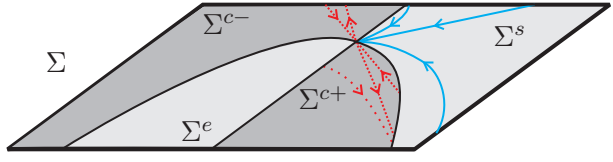}}\caption{In red, the dynamics of $\varphi_{Z_{\lambda}}$ and in blue the local dynamics of $Z_{\lambda}^{\Sigma}$, under the hypothesis $H_1, \dots, H_4$ with $\lambda>0$.}\label{dobra-cuspide-primeiro-deslizante-Lpos}
\end{figure}
In this case, we get that $Z_{\lambda}$ is asymptotic stable at origin, under the hypothesis $H_1, \dots, H_4$.

\vs

In case $\lambda<0$, the Lemma \ref{lema tangencia E2} provides that the trajectories of the sliding vector field $Z_{\lambda}^{\Sigma}$ have a transient behavior in $\Sigma^s$. In fact, in the present case, we shall prove that $Z_{\lambda}$ is not Lyapunov stable at the origin (a fold-fold singularity).

%
%\marginpar{ver isso! Qual tipo de bif. q pode acontecer?}
%This change of stability observed at the origin while the parameter $\lambda$ assume positive and negative values represent a kind of bifurcation, like a \emph{Hopf-bifurcation},  that can occur in PSVF.

\begin{lemma}\label{lema imagem parabola lambda negativo}
Given $p_0=(x_0, -x^{2}_{0},0)$ (under the curve $y=-x^2$), with $x_0>0$, we get
$$
\varphi_{Z_{\lambda}}(x_0, -x^{2}_{0},0) =  \left( 2 x_0 + \frac{3 \lambda}{2 a}, -x^{2}_0 - \frac{3 \lambda(\lambda + 2 a x_0)}{2 a^2} + \frac{d (3 \lambda + 4 a x_0)}{ac},0 \right).
$$
\end{lemma}
\begin{proof}
Straightforward.
\end{proof}

We denote $\varphi_{Z_{\lambda}}(p_0)=p^{\lambda}_1=(x^{\lambda}_1,y^{\lambda}_1,0)$, that can be situated at $\overline{\Sigma^{c+}}$ and in this case, by Lemma \ref{lema dinamica-primeiro-retorno-Ldifzero}, its distance to the origin increase when compared with $p_0$. Otherwise, $p^{\lambda}_1$ can be situated at $\Sigma^{s}$ and in this case the trajectory by this point slides to the parabola $y=-x^2$. The intersection point will be called $p^{\lambda}_2=(x^{\lambda}_2,y^{\lambda}_2,0)=(x^{\lambda}_2,-[x^{\lambda}_2]^2,0)$. As the origin is an attractor for $Z_{\lambda}^{\Sigma}$ we have to answer if the attractiveness $Z_{\lambda}^{\Sigma}$ is greater or less than the repulsiveness of the first return map $\varphi_{Z_{\lambda}}$.

Denote by $d(p,0)$ the euclidian distance between the point $p$ to the origin $0$.

\begin{lemma}\label{lema reta tangente e primeiro retorno}
Under the hypothesis $H_1, \dots, H_4, \lambda<0$ and with the previous notation,
$$
d(p^{\lambda}_2,0)>d(p_0,0).
$$
\end{lemma}
\begin{proof}
By \eqref{equacao campo normalizado com lambda} we obtain a vectorial equation of the straight line
\[
r=\{(x(\alpha),y(\alpha),0)| x(\alpha)=x_0 + \alpha a x_0, y(\alpha)=-x_0^2+\alpha \lambda x_0, \mbox{ with }\alpha \in \R\},
\]tangent to the trajectory of $Z_{\lambda}^{\Sigma}$ by $p_0=(x_0, -x^{2}_{0},0)$.

Note that $r$ splits $\Sigma^s$ in two regions, denoted by $V^+$ and $V^-$, see Figure \ref{dobra-cuspide-primeiro-deslizante-Lneg}. Consider the vertical straight line $s: p = p^{\lambda}_1 + \beta (0,1,0)$, with $\beta \in \R$, see Figure \ref{dobra-cuspide-primeiro-deslizante-Lneg}. Some calculus show that $r \cap s = p_3^{\lambda}$, where $p_3^{\lambda}= (x_3^{\lambda},y_3^{\lambda},0)=\Big(x^{\lambda}_1, -x^{2}_0 + \frac{\lambda}{a}\Big(  \frac{3 \lambda}{2 a}  + x_0 \Big),0 \Big)$. Comparing $y_3^{\lambda}$ with  $y^{\lambda}_1$ we get $y^{\lambda}_1 < y_3^{\lambda}$. Therefore $p^{\lambda}_1$, and consequently $p^{\lambda}_2$, are situated at the region $V_{-}$ described in Figure \ref{dobra-cuspide-primeiro-deslizante-Lneg}. So $d(p^{\lambda}_2,0)>d(p_0,0)$.
\end{proof}
\begin{figure}[ht]
\epsfysize=2.7cm \psfrag{a}{$\Sigma^{c-}$}\psfrag{b}{$\Sigma^{c+}$}\psfrag{c}{$\Sigma^{e}$}\psfrag{d}{$\Sigma^{s}$}\psfrag{S}{$\Sigma$} \psfrag{1}{$(a)$}\psfrag{2}{$(b)$}\psfrag{P}{$p_0$}\psfrag{P1}{$p_1^{\lambda}$}\psfrag{P2}{$p_2^{\lambda}$}
\psfrag{P3}{$p_3^{\lambda}$}\psfrag{R}{$r$}\psfrag{S}{$s$}\psfrag{Vp}{$V^+$}\psfrag{Vn}{$V^-$}
\centerline{\epsfbox{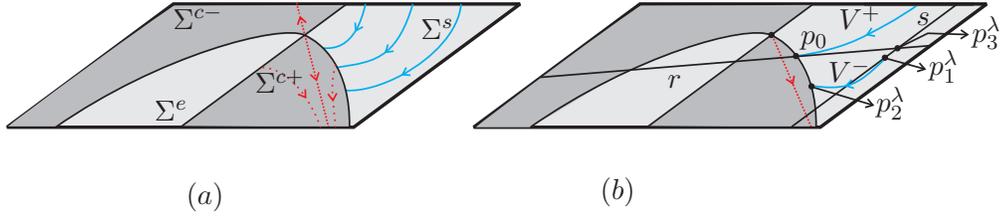}}\caption{In $(a)$ we have the local dynamics of $\varphi_{Z_{\lambda}}$ (in red) and $Z_{\lambda}^{\Sigma}$ (in blue). In $(b)$ is presented the straight lines $r$ and $s$, the points $p_0,p_1^{\lambda}, p_2^{\lambda}$ and $p_3^{\lambda}$ and the regions $V^+$ and $V^-$.}\label{dobra-cuspide-primeiro-deslizante-Lneg}
\end{figure}

\begin{lemma}\label{lema primeiro retorno lambda negativo}
$Z_{\lambda}$ is not Lyapunov stable at origin for $\lambda<0$.
\end{lemma}
\begin{proof}
By Lemma \ref{lema tangencia E2} we get that $Z_{\lambda}^{\Sigma}$ has a transient behavior, i.e., $\Sigma^{c+}$ is a attractor set for $Z_{\lambda}^{\Sigma}$ and by Lemma \ref{lema dinamica-primeiro-retorno-Ldifzero} we conclude that all points in $\Sigma^{c+}$ converge to $\overline{\Sigma^s}$. So, to analyze the stability of $Z_{\lambda}$ at origin it is sufficient study the iterates of $Z_{\lambda}$ at boundary of $\Sigma^s$. By Lemma \ref{lema reta tangente e primeiro retorno} we obtain that the distance between the origin and a point in $\overline{\Sigma^s}$ increases along the time. Therefore, we conclude that $Z_{\lambda}$ is not Lyapunov stable at origin for this case.
\end{proof}

\begin{remark}
As consequence of Lemmas \ref{lema tangencia E2} and Lemma \ref{lema dinamica-primeiro-retorno-Ldifzero} we get that $Z_0$ has codimension at least two because the eigenspaces of the normalized sliding vector field and the first return map are tangent to $S_X$.

\end{remark}

\section{Proof of main results} \label{main results}

\subsection{Case $\lambda = 0$}

When $\lambda=0$, by Proposition \ref{proposicao a trajetoria cai no sliding}, the trajectories of all points in $\R^3$ intersect $\overline{\Sigma^s}$. By hypotheses $H_2$ and $H_4$, the omega limit set of all trajectories in $\overline{\Sigma^s}$ is the origin. So, $Z_0$ is asymptotically stable.

\subsection{Case $\lambda > 0$}

When $\lambda>0$, by item (a) of Lemma \ref{lema dinamica-primeiro-retorno-Ldifzero} the trajectories of all points in $\R^3$ intersect $\overline{\Sigma^s}$. Since the origin is a hyperbolic attractor for $\lambda=0$, the same holds when $\lambda \neq 0$, sufficiently small. By Lemma \ref{lema tangencia E2} we get $E^{\lambda}_1\subset \Sigma^c$ and
 $E^{\lambda}_2 \subset \Sigma^s$, for $x>0$. So, $Z_{\lambda}$ is asymptotically stable.

\subsection{Case $\lambda < 0$}

When $\lambda<0$, the result is an immediate consequence of Lemma \ref{lema primeiro retorno lambda negativo}.

\vspace{1cm}

\noindent {\textbf{Acknowledgments.}   The first author is
partially supported by a FAPESP-BRAZIL grant
2012/00481-6. This work is partially realized at UFG as a part of project numbers 35796 and
35797.}


\begin{thebibliography}{99}

\bibitem{Andronov} {\sc A. Andronov and S. Pontryagin}, {\it Structurally stable systems}, Dokl.
Akad. Nauk SSSR \textbf{14} (1937), 247--250.


\bibitem{diBernardo-livro} {\sc M. di Bernardo, C.J. Budd, A.R. Champneys and P. Kowalczyk},
{\it Piecewise-smooth Dynamical Systems $-$ Theory and
Applications}, Springer-Verlag (2008).

%\bibitem{diBernardo-1} {\sc M.
%di Bernardo, C.J. Budd, A.R. Champneys, P. Kowalczyk, A.B. Nordmark,
%G.O. Tost and P.T. Piiroinen}, {\it Bifurcations in nonsmooth
%dynamical systems}, SIAM Rev. \textbf{50} (2008), 629--701.


\bibitem{diBernardo-electrical-systems} {\sc M. di Bernardo, A. Colombo and E. Fossas}, {\it Two-fold singularity in
nonsmooth electrical systems}, Proc. IEEE International Symposium on
Circuits ans Systems (2011), 2713--2716.



\bibitem{Jeffrey-T-sing} {\sc M. di Bernardo, A. Colombo,E. Fossas and M.R. Jeffrey},
{\it Teixeira singularities in 3D switched feedback control
systems}, Systems and Control Letters \textbf{59} (2010), 615--622.

%
%\bibitem{diBernardo-2} {\sc M. di Bernardo, A.R. Champneys, S.J. Hogan, M. Homer, P. Kowalczyk,
%Yu.A. Kuznetsov,  A.B. Nordmark and P.T. Piiroinen}, {\it
%Two-parameter discontinuity-induced bifurcations of limit cycles:
%Classification and open problems}, Internat. J. Bifur. Chaos Appli.
%Sci. Engrg. \textbf{16} (2006), 601--629.

%\bibitem{diBernardo-3-Jeffrey} {\sc M. di Bernardo, A. Colombo, S.J. Hogan and M.R. Jeffrey},
%{\it Bifurcations of piecewise smooth flows: perpectives,
%methodologies and open problems}, Phisica D, to appear.

%
%
%\bibitem{B-P-S}  {\sc M.E. Broucke, C.C. Pugh and S.N. Simic},  {\it Structural stability of piecewise smooth systems}, in: The Geometry of Differential Equations and Dynamical Systems, Comput. Appl.
%Math. 20 (1-2) (2001), 51--89.

\bibitem{Eu-fold-cusp} {\sc C.A. Buzzi, T. de Carvalho and M.A. Teixeira},
{\it On $3$-parameter families of piecewise smooth vector fields in
the plane}, SIAM J. Applied Dynamical Systems, 11(4) (2012), 1402--1424.

\bibitem{Eu-fold-sela} {\sc C.A. Buzzi, T. de Carvalho and M.A. Teixeira},
{\it On three-parameter families of Filippov systems $-$ The
Fold-Saddle singularity}, International Journal of Bifurcation and
Chaos, vol 22, No. 12, 1250291.

%\bibitem{Eu-Durval-Belgian} {\sc T. de Carvalho and D.J. Tonon}, {\it Generic
% Bifurcations of Planar Filippov Systems via Geometric Singular Perturbations},
% Bulletin of the Belgian Mathematical Society Simon Stevin,  18 (2011),
%861–-881.


%\bibitem{Coll-Gasull-Prohens} {\sc B. Coll, A. Gasull and R. Prohens},
%{\it Center-focus and isochronous center problems for discontinuous
%differential equations}, Discrete and Continuous Dynamical Systems
%\textbf{6} (2000), 609--624.


\bibitem{Jeffrey-colombo} {\sc  A. Colombo and M.R. Jeffrey},  {\it The two-fold singularity of
discontinuous vector fields}, SIAM J. Appl. Dyn. Syst. 8 (2009),
624--640.

\bibitem{Jeffrey-colombo-2011} {\sc  A. Colombo and M.R. Jeffrey}, {\it Non-deterministic chaos, and the two fold singularity in piecewise smooth flows}, SIAM J. Appl. Dyn. Syst. 10 (2011), 423--451.

\bibitem{Fi} {\sc A.F. Filippov}, {\it Differential Equations with Discontinuous Righthand Sides},
Mathematics and its Applications (Soviet Series), Kluwer Academic
Publishers-Dordrecht, 1988.
%
%\bibitem{Glendinning} {\sc P. Glendinning},
%{\it Non-smooth pitchfork bifurcations}, Discrete and Continuous
%Dynamical Systems Ser. B \textbf{4} (2004), 457--464.

\bibitem{Marcel} {\sc M. Guardia, T.M. Seara and M.A. Teixeira}, {\it Generic bifurcations of low codimension of planar
Filippov Systems}, Journal of Differential Equations \textbf{250}
(2011) 1967--2023.


\bibitem{J-T-T1} {\sc A. Jacquemard, M.A. Teixeira and D.J. Tonon}, {\it Stability conditions in piecewise smooth dynamical systems at a
two-fold singularity}, Journal of Dynamical and Control Systems, vol. \textbf{19}, 47-67, 2013.


\bibitem{J-T-T2} {\sc A. Jacquemard, M.A. Teixeira and D.J. Tonon}, {\it Piecewise smooth reversible dynamical systems at a two-fold
singularity}, International Journal of Bifurcation and Chaos, vol. \textbf{22}, 2012.


\bibitem{Jeffrey-1} {\sc M.R. Jeffrey}, {\it Two-folds in nonsmooth dynamical systems}, IFAC Chaos09 proceedings, Queen Mary June 2009.

%
%\bibitem{K} {\sc V. S. Kozlova},
%{\it Roughness of a Discontinuous System}, Vestinik Moskovskogo
%Universiteta, Matematika \textbf{5} (1984), 16--20.


\bibitem{Kuznetsov} {\sc YU. A. Kuznetsov, S. Rinaldi and
A. Gragnani}, {\it One-Parameter Bifurcations in Planar Filippov
Systems}, Int. Journal of Bifurcation and Chaos, \textbf{13} (2003),
2157--2188.

\bibitem{Lamb-Makarenkov} {\sc O. Makarenkov and J.S.W. Lamb }, {\it Dynamics and bifurcations of nonsmooth systems: A survey}, Physica D: Nonlinear Phenomena
 \textbf{241}, Issue 22, (2012), 1826–-1844


%
%\bibitem{Ponce-Pagano-IFAC} {\sc D. J. Pagano and
%E. Ponce}, {\it Sliding Dynamics Bifurcations in the Control of
%boost converters}, 18th IFAC World Congress, Milano, Italy, August
%28th - September 2nd (2011), 13293--13298.


\bibitem{Soto-GenericOneParameterFam} {\sc J. Sotomayor},
{\it  Generic one-parameter families of vector fields on
two-dimensional manifolds}, Inst. Hautes Études Sci. Publ. Math.,
\textbf{43} (1974), 5--46.

%
%\bibitem{Soto-Ana} {\sc J. Sotomayor, A.L. Machado},
%{\it  Structurally stable discontinuous vector fields on the plane},
%Qual. Theory Dyn. Syst., \textbf{3} (2002), 227--250.

%
%\bibitem{Soto-Teixeira-fronteira} {\sc J. Sotomayor and M.A.
%Teixeira}, \textbf{Vector fields near the boundary of a
%$3-$manifold}, Lect. Notes in Math., \textbf{331}, Springer Verlag,
%(1988),169--195.
%
%\bibitem{ST} {\sc J. Sotomayor and M.A.
%Teixeira}, {\it Regularization of  Discontinuous Vector Fields},
%International Conference on Differential Equations, Lisboa (1996),
%207--223.

%
%\bibitem{T1} {\sc M.A. Teixeira},
%{\it  Generic bifurcation in manifolds with boundary}, Journal of
%Differential Equations \textbf{25} (1977), 65--88.


\bibitem{Marco-bollettino} {\sc M.A. Teixeira},
{\it  Generic bifurcation of certain singularities}, Bollettino
della Unione Matematica Italiana (5), \textbf{16-B} (1979),
238--254.

\bibitem{Marco-GenericBifSlidingVecFields} {\sc M.A. Teixeira},
{\it  Generic bifurcation of sliding vector fields}, Journal of
Mathematical Analysis and Aplications \textbf{176} (1993), 436--457.



\bibitem{teixeira-stability} {\sc M.A. Teixeira},
{\it Stability conditions for discontinuous vector fields}, Journal
of Differential Equations \textbf{88} (1990), 15--29.

%\bibitem{T} {\sc M.A. Teixeira},
%{\it  Generic singularities of discontinuous vector fields}, An. Ac.
%Bras. Cienc. \textbf{53} (1991), 257--260.

\bibitem{Marco-enciclopedia} {\sc M.A. Teixeira},
{\it  Perturbation Theory for Non-smooth Systems}, Meyers:
Encyclopedia of Complexity and Systems Science \textbf{152}  (2008).

%
%\bibitem{ZKB} {\sc Y. Zou, T. K\"{u}pper and W.J. Beyn},
%{\it Generalized Hopf Bifurcation for Planar Filippov Systems
%Continuous at the Origin}, J. Nonlinear Sci, Vol 16 (2006),
%159--177.
%
%\bibitem{vishik} {\sc S.M. Vishik},
%{\it Vector fields near the boundary of a manifold}, Vestinik
%Moskovskogo Universiteta. Matematika \textbf{27} (1972), 21--28.

%%%%%%%%%%%%%%%%%%%%%%%%%%%%%%%%%%%%%%%%%%%%%%%%%%%%%%%%%%%%%%%%%%%%%%%%%%%%%%%%%%%%%%










\end{thebibliography}
\end{document}